\documentclass{amsart}
\usepackage{amssymb}
\usepackage{amsmath}
\usepackage{amscd}
\usepackage{amsthm}
\usepackage{tikz}

\newtheorem{proposition}{Proposition}[section]
\newtheorem{theorem}[proposition]{Theorem}
\newtheorem{lem}[proposition]{Lemma}

\theoremstyle{remark} 
\theoremstyle{definition}
\newtheorem{exam}[proposition]{Example}
\newtheorem{rem}[proposition]{Remark}

\numberwithin{equation}{section}

\newcommand{\Z} {\mathbb{Z}}
\newcommand{\N} {\mathbb{N}}
\newcommand{\R} {\mathbb{R}}

\newcommand{\Q} {\mathbb{Q}}

\newcommand{\B} {\mathbb{B}}

\begin{document}

\title{Generic point equivalence and Pisot numbers}
\author{Shigeki Akiyama, Hajime Kaneko and Dong Han Kim}
\date{\today}



\maketitle

\begin{abstract}
Let $\beta >1$ be an integer or generally a Pisot number.
Put $T(x) = \{ \beta x \}$ on $[0,1]$ and let $S: [0,1]\to [0,1]$ be a piecewise linear transformation 
whose slopes have 
the form $\pm \beta^m$ with positive integers $m$. 
We give sufficient conditions that  
 $T$ and $S$ have the same generic points. 
\end{abstract}

\section{Introduction}

Let $b\geq 2$ be an integer and $T : [0,1] \to [0,1]$ a map given by $T(x) = \{b x\}$, 
where $\{x\}$ denotes the fractional part of $x$.
A real number $x \in [0,1]$ is called to be normal in base-$b$ 
if in the base-$b$ expansion of $x$ any pattern of length $L$ appears with relative frequency tending to $b^{-L}$. 
Wall~\cite{Wal} showed that $x$ is normal in base-$b$ if and only if $x$ is a $T$-generic point, i.e., 
its orbital points $x, T(x), T^2 (x), \dots $ distribute uniformly. 
We recall that nonzero integers $m$ and $n$ are multiplicatively dependent if there exists 
$(i,j)\in \Z^2\backslash \{(0,0)\}$ satisfying $m^i n^j= 1$. 
Maxfield~\cite{Max} proved that if two positive integers $b_1, b_2$ are multiplicatively dependent, then base-$b_1$ normality is equivalent to base-$b_2$ normality. 
Schweiger~\cite{Sch} and Vandehey~\cite{Van} showed that if two number theoretic transformations $T$ and $S$ satisfy $T^m = S^n$ 
for some positive integers $n,m$, then every $T$-normality is equivalent to $S$-normality.
Kraaikamp and Nakada~\cite{KN} 
gave counter examples that the other direction does not hold.
They used the jump transformation to show the equivalence of normality: {\bf normality equivalence}, in short. 

In this article, we relax a sufficient condition for normality equivalence and obtain 
infinite families of examples (see Examples \ref{FlipDecimal} and \ref{Luroth}). 
Moreover, we shall generalize the concept of normality equivalence to include systems whose invariant measures may be different.
Let $(X,\B,\mu,T)$ and $(X,\B,\nu, S)$ be two ergodic measure preserving systems 
with a common underlying space $X$. 
We assume that $X$ is a compact metric space,  
$\B$ is the sigma-algebra of Borel sets in $X$, and that $\mu,\nu$ are probability measures. 
A point $x\in X$ is called $T$-generic if $\lim_{N \to \infty} \frac 1N \sum_{n=0}^{N-1} f(T^n x) = \int_X f d\mu$  
 for any continuous function $f$ on $X$. 
We say that  $S$ and $T$  are {\bf generic point equivalent} if the set of $S$-generic points coincide with the set of $T$-generic points.
The main purpose of this paper is to give sufficient conditions for generic  point equivalence for $X=[0,1]$, using the Pyatetskii-Shapiro criterion. 

Let $\beta$ be a {\bf Pisot number}: a  real algebraic integer greater than one
whose Galois conjugates (except itself) have modulus less than one.
Note that any integer greater than 1 is a Pisot number. 
Put $T(x) = \{ \beta x \}$ on $[0,1]$.
Let $S: [0,1]\to [0,1]$ be a piecewise linear transformation. 
In Section 3, we give a sufficient condition for 
generic point equivalence of $S$ and $T$ in the case where 
the slopes of $S$ have the form $\pm \beta^m$ with positive integers $m$.
More precisely, we show that if $S$ admits an absolutely continuous invariant measure and the invariant density is bounded above and away 
 from 0 and all intercepts are in $\Q(\beta)$, 
then $T$ and $S$ are generic point equivalent.
In Section 2, we give Proposition~\ref{prop:2-2}, which is applicable to prove generic point 
equivalence. 
Using this proposition, we shall prove our main result. 

The Pisot slope condition is essential: our proof  depends on the structure of the point set generated by Pisot numbers.
The proof becomes simpler than those in literature 
and applicable to a wide class of piecewise linear maps.
In fact, we require no condition on the position of discontinuities. 
In particular, we provide a one parameter 
family of maps (the cardinality of the maps is uncountable) by continuously 
shifting the discontinuity so that 
all the maps in the family are generic point equivalent (see Example \ref{ACIM}).  This appears to be 
the first result on generic point equivalence among generically non-Markov piecewise linear maps.
 
\section{Criteria for generic point equivalence}

We now review the Pyatetskii-Shapiro criterion. 
Let $(X,\B,\mu, T)$ be an ergodic measure preserving system.
Denote the characteristic function of $V\in \B $ by $\chi_{V}$  
and the set of continuous functions on $X$ by $C(X)$.
Let $\mathcal{C}\subset \B$ be a semi-algebra generating $\B$ in the sense that the minimal sigma algebra including $\mathcal{C}$ is $\B$. 
Then the Pyatetskii-Shapiro criterion reads 

\begin{theorem}[{\cite{Pos}, Theorem 6}]\label{thm:1-1}
Let $(X,\B,\mu, T)$ be an ergodic measure preserving system.   
Let $x_0\in X$ and $\mathcal{C}\subset \B$ be a semi-algebra generating $\B$. 
We assume that any function $f\in C(X)$ is a limit point of  the set of the (finite) linear combinations of the characteristic functions of $V\in\mathcal{C}$ with respect to the sup norm.  
Suppose that there exists a positive constant $C$ satisfying 
\begin{equation}\label{eqn:1-1}
\limsup_{N\to\infty}\frac{1}{N}\sum_{n=0}^{N-1}\chi_{I}(T^{n}x_0)\leq C \mu(I)  
\end{equation}
for any $I\in \mathcal{C}$. Then, $x_0$ is a $T$-generic point.
\end{theorem}

We now introduce a criterion for generic point equivalence deduced from Theorem~\ref{thm:1-1}. 
\begin{proposition}\label{prop:2-2}
Let $([0,1],\B,\mu,T)$ and $([0,1],\B,\nu, S)$ be two ergodic measure preserving systems. Let 
$\mathcal{C}$ be a semi-algebra generating $\B$. 

Let $x_0\in [0,1]$ be a $T$-generic point. 
Suppose that there exist a positive integer $M$, a positive real number $C$, and a sequence $(k(n))_{n=0}^{\infty}$ 
of nonnegative integers satisfying the following: 
\begin{enumerate}
\item For any nonnegative integer $m$, we have 
\[\mbox{Card}\hspace{0.3mm}\{n\geq 0\mid k(n)=m\}\leq M,\]
where Card denotes the cardinality. 
\item For any $n\geq 0$, we have 
\[k(n)\leq M \cdot \max\{1,n\}.\]
\item Let $I\in \mathcal{C}$. Then there exists $\widetilde{I}=\cup_{i=1}^{r}\widetilde{I_i}$, where 
$\widetilde{I_1},\ldots,\widetilde{I_r}$ are subintervals of $[0,1]$, such that 
\[
\mu(\widetilde{I})\leq C \nu(I)
\]
and that, for any $n\geq 0$, 
\[\mbox{if }S^n x_0\in I\mbox{, then }T^{k(n)}x_0\in \widetilde{I}.\]
\end{enumerate}
Then $x_0$ is an $S$-generic point. 
\end{proposition}
\begin{proof}
Let $I\in\mathcal{C}$ and $N$ be an integer greater than 1. Put 
\begin{equation*}
\rho(N):=\max\{k(n)\mid 0\leq n\leq N-1\} \leq M(N-1)\leq MN-1.
\end{equation*}
Then, we see
\begin{align*}
\frac{1}{N}\sum_{n=0}^{N-1}\chi_{I} (S^n x_0)&=
\frac{1}{N}\sum_{m=0}^{\rho(N)}\sum_{\substack{k(n)=m \\0\leq n\leq N-1}}\chi_{I} (S^n x_0)\\
&\leq 
\frac{1}{N}\sum_{m=0}^{\rho(N)}\sum_{\substack{k(n)=m \\0\leq n\leq N-1}}
\chi_{\widetilde{I}} (T^m x_0)\\
&\leq 
\frac{M}{N}\sum_{m=0}^{\rho(N)}
\chi_{\widetilde{I}} (T^m x_0)
\leq 
M^2\cdot 
\frac{1}{MN}\sum_{m=0}^{MN-1}
\chi_{\widetilde{I}} (T^m x_0).
\end{align*}
Since $x_0$ is $T$-generic, we get 
\[
\limsup_{N\to\infty}\frac{1}{N}\sum_{n=0}^{N-1}\chi_{I} (S^n x_0)
\leq M^2\mu(\widetilde{I})\leq M^2C \nu(I),
\]
which implies by Theorem \ref{thm:1-1} that $x_0$ is $S$-generic. 
\end{proof}
\begin{rem}
Proposition~\ref{prop:2-2} can be generalized for two ergodic measure preserving systems
$([0,1]^d,\B,\mu,T)$ and $([0,1]^d,\B,\nu, S)$.
\end{rem}

\section{Pisot slope condition}
Let $\N$ be the set of positive integers. 
Given $\beta>1$, let $T (x) = \{ \beta x \}$ be a map on $[0,1]$. 
Then $T$ is ergodic with respect to a unique 
absolutely continuous invariant measure $\mu_\beta$ whose 
density is bounded and away from 0,  (see \cite{Par}). 
Let $[0,1]=\cup_{i=1}^{\ell} J_i$ be a finite partition
of $[0,1]$ into subintervals\footnote{Subinterval $J_i$ can be closed, open or semi-open, even a singleton.}.
Let $S : [0,1] \to [0,1]$ be a transformation given by
\begin{align*}
S(x) = \epsilon_i \beta^{m_i} x + b_i, \qquad \text{ for } \ x \in J_i,
\end{align*}
where $\epsilon_i \in \{ -1, 1\}$, $m_i \in \mathbb N$ and $b_i \in \Q(\beta)$ for $1 \le i \le \ell$. 

For any $x\in [0,1]$ and $h\geq 0$, let $i(h)=i(x_0;h)$ be defined by $ S^h (x) \in J_{i(h)}$.
Then we have, for any $n\geq 0$, 
\begin{align}\label{yyy}
S^n(x)=  \left(\prod_{h=0}^{n-1} \epsilon_{i(h)}  \right)
\beta^{\sum_{h=0}^{n-1} m_{i(h)}}  x + \sum_{j=0}^{n-1} \left(\prod_{h >j}^{n-1} \epsilon_{i(h)} \right) \beta^{\sum_{h >j}^{n-1} m_{i(h)}}  b_{i(j)}.
\end{align}
Put 
\begin{align}\label{theta}
\theta_n(x_0):=\sum_{h=0}^{n-1} m_{i(h)}, 
\end{align}
where $\beta^{\theta_n(x_0)}$ gives the absolute value of the slope of $S^n$ at $x_0$.
Hereafter, unless it is stated explicitly,  we assume that $\beta$ is a Pisot number.

A subset $Y$ of $\R$ is {\it uniformly discrete} if there exists a positive constant $R$ such that for  any two distinct points $y, y' \in Y$, we have $|y-y'|>R$.

\begin{lem}\label{lem:Pisot}
Let $E$ be a finite subset of $\Q(\beta)$ and put
$$F_E := \left \{\left. \sum_{j=0}^r d_j \beta^j \, \right| \, d_j \in E, \, r = 0, 1, 2, \dots \right \}.
$$
Then $F_E$ is uniformly discrete.
\end{lem}

This follows from a standard discussion (e.g. Garcia \cite{Gar}), but we show it for completeness.

\begin{proof}
Without loss of generality, we may assume that $0\in E$. 
We claim that $0$ is not an accumulation point of $F_E$. 
In fact, let $\beta^{(j)}$ be the Galois conjugates of $\beta$ for $j = 1,\dots, d$
with $\beta^{(1)}=\beta$. Take a positive integer $L$ such that $E \subset \frac 1L \Z[\beta]$. 
Suppose that $0\ne \sum_{j=0}^r d_j \beta^j\in F_E$. 
Considering the image of the Galois conjugate map $\phi_i$ which sends $\beta$ to $\beta^{(i)}$, we obtain
$$
\left| L^d \prod_{i=1}^{d} \sum_{j=0}^{r} \phi_i(d_j) (\beta^{(i)})^j\right|\ge 1
$$
because the product must be an integer. 
Since $\beta$ is a Pisot number, we obtain
$$
\left|\sum_{j=0}^r d_j \beta^j\right|\ge \frac 1{L^d}\prod_{i=2}^d \left(\frac {A_i}{1-|\beta^{(i)}|}\right)^{-1},
$$
where $A_i=\max \{|\phi_i(d)|\ |\  d \in E\}$ is a positive constant because $E$ is a finite set. 
This shows the claim. 
Note that $F_E-F_E=F_{E-E}$ by $0\in E$. 
By the same proof replacing $E$ by $E-E$, we obtain the assertion.
\end{proof}

If $S$ and $T$ are generic point equivalent, then the set of non-generic points of $T$ and that of $S$ 
are identical. Thus we may expect that eventually periodic orbits of $T$ and those of $S$ coincide.
Next theorem confirms this expectation that $T$ and $S$ share the same set of eventually periodic orbits.

\begin{theorem}
The orbit $S^n(x)$ for $n=0,1,\dots$ is eventually periodic if and only if $x\in \Q(\beta)$.
\end{theorem}

\begin{proof}
Because $b_i \in \Q(\beta)$, 
every eventually periodic point of $S$ belongs to $\Q(\beta)$.
Assume that $x\in \Q(\beta)$. 
Take a positive integer $L$ such that $Lx$ and $Lb_j$ are in $\Z[\beta]$.
Then for all $n \ge 0$ we have $L S^n(x)\in \Z[\beta]$ by (\ref{yyy}).
Since $\beta$ is a Pisot number, for each $i=1,\dots,d$ there is a constant $C_i>0$ such that
$|\phi_i(L S^n(x))|\le C_i$  for all $n \ge 0$.
Since the image of the Minkowski embedding of $\Z[\beta]$ forms a lattice in $\R^d$, the orbit is eventually periodic. 
\end{proof}

Now we are in position to state our main theorem.

\begin{theorem}\label{thmTS}
Let $T$, $S$ be the maps defined above. 
Suppose that $S$ preserves a probability measure $\nu$, which is ergodic and 
absolutely continuous with respect to the Lebesgue measure $\lambda$. 
Moreover, assume that there exists a positive constant $c$ satisfying 
\begin{equation}\label{Positive}
c^{-1}\lambda(E)\leq \nu(E)\leq c \lambda(E)
\end{equation}
for any Borel set $E\subset [0,1]$.  
Then $T$ and $S$ are generic point equivalent.
\end{theorem}

The condition \eqref{Positive} implies that $\nu$ and $\lambda$ are equivalent. 
Kowalski \cite{Kow} showed under ergodicity of $S$ that the converse holds as well in this setting.

\begin{proof}
If necessary, changing the constant $c$, we may assume that 
$$c^{-1} \lambda(E)\leq \mu_{\beta}(E)\leq c \lambda (E)$$
for any Borel set $E\subset [0,1]$ (Parry \cite{Par}, 
Ito-Takahashi \cite{IT}).
Let $$E = \{\pm b_1, \pm b_2 , \dots , \pm b_k \} \cup \{0, 1, 2, \dots , \lfloor \beta \rfloor \}.$$
Putting $F : = F_{E-E} = F_E - F_E$, we get by Lemma~\ref{lem:Pisot} that $F$ is uniformly discrete.

First we assume that $x_0 \in [0,1]$ is a $T$-generic point.
For each $n \ge 0$,  let $k(n):=\theta_n(x_0)$ be defined by \eqref{theta}. 
Then we see that 
$$S^{n}(x_0)=\epsilon(\beta^{k(n)} x_0 -b), \qquad T^{k(n)}(x_0)=\beta^{k(n)}  x_0 -b',$$
for some  $\epsilon\in\{1,-1\}$ and $b,b'\in F_E$. 
We now verify that $(k(n))_{n=0}^{\infty}$ and $M:=\max\{m_1,\ldots,m_{\ell}\}$ satisfy the assumptions of 
Proposition~\ref{prop:2-2}, where $\mathcal{C}$ is the set of subintervals of $[0,1]$. 
The first and the second assumptions are clear by 
$1\leq k(n+1)-k(n)\leq M$ for any $n\geq 0$. For any interval $I$, put 
$$\tilde I =\left( \bigcup_{t \in F\cap [-1,2]} \Big( \big( I + t \big) \cup \big( - I + t \big) \Big)\right)\cap [0,1].$$
Then we have 
$$
\mu_{\beta}( \tilde I ) \le 2 c \, \textrm{Card} (F\cap[-1,2]) \lambda(I)
\le 2 c^2 \, \textrm{Card} (F\cap[-1,2]) \nu(I).
$$
We now assume for $n\ge 0$ that $S^n(x_0)\in I$. Noting that 
$$
b-b'=T^{k(n)}(x_0)-\epsilon S^n(x_0)\in [-1,2]\cap F,
$$
we obtain 
$$
T^{k(n)}(x_0)=\epsilon S^n(x_0)+(b-b')\in \tilde I.
$$
Hence, $x_0$ is $S$-generic by Proposition~\ref{prop:2-2}.

We prove the other direction. Let $x_0 \in [0,1]$  be an $S$-generic point.
For each $n\geq 0$, we define $k(n)$ by 
$$
k(n) := \max \{ k \, | \,  \theta_k(x_0) \le \beta^n \}.
$$
For any $h \geq 0$, we see that $k(n)= h$ if and only if 
\begin{align}\label{xxx}
\theta_{h}(x_0)\leq n<\theta_{ h+1}(x_0)=\theta_{h}(x_0)+m_{i(h)}.
\end{align}
Moreover, we see for any $n\geq 0$ that 
\begin{align}\label{xxxyyy}
\theta_{k(n)}(x_0)= \beta^{n-j},
\end{align}
for some $0\leq j<M=\max\{m_1,\ldots,m_\ell\}$. In what follows, we show 
that $(k(n))_{n=0}^{\infty}$ and $M$ satisfy the assumptions of 
Proposition~\ref{prop:2-2}. The first and the second assumptions are clear by \eqref{xxx} and 
$0\leq k(n+1)-k(n)\leq 1$ for any $n\geq 0$.
For any interval $I\subset [0,1]$, put 
$$\tilde I 
= \left(\bigcup_{j=0}^{M-1} \bigcup_{t \in F\cap[-1,2]} \Big( \big( T^{-j}(I) + t \big) \cup \big( - T^{-j} (I) + t \big) \Big)\right)
\cap [0,1].$$
Then we get 
\begin{align*}
\nu ( \tilde I ) 
&\le 2 c \, \textrm{Card} ( F \cap[-1,2] ) M \max_{0\le j\le M-1}\lambda \left(T^{-j}(I)\right)\\
&\le 2 c^2 \, \textrm{Card} ( F \cap[-1,2] ) M \mu_{\beta} (I).
\end{align*}
Suppose for $n\geq 0$ that $T^n(x_0)\in I$. Let $j$ be defined by \eqref{xxxyyy}. 
In the same way as the former part of the proof of Theorem~\ref{thmTS}, we get 
$$\epsilon T^{n-j} ( x_0) + S^{k(n)} (x_0)\in F \cap[-1,2],$$
for some $\epsilon\in \{1,-1\}$. Therefore, we deduce that 
\begin{equation*}
S^{k(n)} (x_0)=-\epsilon T^{n-j} ( x_0)+\left(\epsilon T^{n-j} ( x_0) + S^{k(n)} (x_0)\right)\in \tilde I. \qedhere
\end{equation*}
\end{proof}

\begin{rem}
It is natural to assume that all slopes in modulus are certain powers of a fixed number, 
since we can not expect generic point equivalence for multiplicatively independent slopes. 
Indeed, if $a$ and $b$ are multiplicatively independent positive integers, then 
Schmidt \cite{WSch} showed that there are uncountably many $a$-normal numbers which are not $b$-normal. 
Moreover, Pollington \cite{Poll} calculated the Hausdorff dimension of such numbers. 
Consider a partition of the set $\{2,3,\ldots\}$ into $A$ and $B$ so that all multiplicatively dependent integers fall into the same class. 
Then the set of real numbers normal in any base from $A$ and in no base from $B$ has Hausdorff dimension 1.
Explicit construction of numbers which are $a$-normal but not $b$-normal  
 is exploited when $a$ divides $b$, e.g., \cite{Wag}, \cite{KS}, \cite{JV}.
However, we do not yet know a concrete example of a $2$-normal number which is not $3$-normal.
\end{rem}

\begin{rem}
Theorem~\ref{thmTS} does not extend to an infinite partition, due to an example by
J\"ager \cite{Jag} for the case of $\beta=10$. 
Let $T=\{10 x\}$ on $[0,1]$ and $x=(.x_1x_2\ldots)$ be 
the coding of $x$ by $T$, i.e., the decimal expansion of $x$.
Let $m$ be the first occurrence of a fixed digit $r\in \{0,1,\dots, 9\}$ that 
$x_m=r$, then we define a jump transform 
$S_r(x):=(.x_{m+1}x_{m+2}\ldots)$. 
If there is no occurrence of $r$, put $S_r(x):=0$. 
Then every $T$-generic point is $S_r$-generic, but the converse does not hold.
\end{rem}

\begin{rem}
We show that the condition (\ref{Positive}) is not preserved after taking flips.
Let $\beta >1$ be a real number and $0 = t_0 < t_1 < \dots < t_k = 1$ is a finite partition of $[0,1]$.
Suppose that $T$ is a map on $[0,1]$ which has slope of $\pm \beta^{m_i}$ on
$[t_{i-1},t_i)$ and has an invariant measure which is equivalent to the Lebesgue measure.
If $S$ is a locally flipped map of $T$ on $[0,1]$, that is, on one interval $[t_{i-1},t_i)$, 
$S$ has the opposite slope $\mp \beta^{m_i}$ and $T((t_{i-1}+t_i)/2)=S((t_{i-1}+t_i)/2)$, 
then one might expect that $S$ 
also has an invariant measure equivalent to the Lebesgue measure. Unfortunately
this is not true. Here is a counter example. Let $1<\beta<\sqrt{2}$ and put
$$
S(x)=\begin{cases} -\beta x+1 & x\in [0,1/\beta), \\
                  \beta x -1 & x\in [1/\beta,1].\end{cases}
$$
The map $S$ is a locally flipped map of the beta transformation $T$ having density away from zero.
Since the dynamics of $S$ on $[\beta-1,-\beta^2+\beta+1]$ is dissipative, 
the density of $S$ on $[\beta-1,-\beta^2+\beta+1]$ is zero. 
The explicit densities of flipped beta expansions are given in Gora \cite{Gora0}.
\end{rem}

\section{Examples}

We apply Theorem~\ref{thmTS} to certain families of piecewise linear maps on $[0,1]$. 

\begin{exam}
\label{FlipDecimal}
Let $r$ be an integer greater than 1. For $\mathbf s=(s_0,s_1,\ldots,s_{r-1})\in \{0,1\}^r$, let 
$T(r, \mathbf s;x):[0,1]\to[0,1]$ be a map defined by 
\begin{align*}
T(r, \mathbf s;x):=
\begin{cases}
s_i+(-1)^{s_i}\{r x\} & \mbox{if }x\in [i/r,(i+1)/r)\\ 
0 & \mbox{if }x=1. 
\end{cases}
\end{align*}
Then $([0,1], \B, \lambda, T(r,s;x))$ is an ergodic measure preserving system, where $\lambda$ is 
the Lebesgue measure. 
Let $q$ be integers greater than 1 and $\mathbf t \in \{0,1\}^q$.  
Assume that $q$ and $r$ are multiplicatively dependent. 
Then $T(q, \mathbf t;x)$ and $T(r, \mathbf s;x)$ are generic point equivalent. 
As a special case, the tent map
$$
f(x)=\begin{cases} 2x & 0\le x <1/2 \\
                   2(1-x) & 1/2\le x\le 1\end{cases}
$$
and the binary expansion map $T(x)=\{2x\}$ are generic point equivalent. 
This simple case already seems new. 
Indeed, this serves an alternative proof of Corollary 19 in \cite{AJKM} which solves several conjectures posed
in \cite{SS}, as the set of $2$-normal numbers lies exactly in the 3rd Borel-hierarchy by \cite{KL}.
\end{exam}

The following examples were shown by Kraaikamp and Nakada in \cite{KN}.

\begin{exam}
\label{Counter}
Consider the maps 
$T_1:[0,1]\to[0,1]$ and $S_1:[0,1]\to [0,1]$ defined by $T_1(x)=\{2x\}$ and 
\begin{align*}
S_1(x):=\begin{cases}
2 x, & x\in [0,1/2),\\
\{ 4 x \}, & x\in[1/2,1].
\end{cases}
\end{align*}
Let $\beta = (\sqrt 5 + 1)/2$. 
Define $T_2 :[0,1]\to[0,1]$ and $S_2 :[0,1]\to [0,1]$ by $T_2 (x)=\{ \beta x\}$ and 
\begin{align*}
S_2 (x):=\begin{cases}
\beta x, & x\in [0,1/\beta),\\
\beta^2 x -\beta, & x\in[1/\beta,1].
\end{cases}
\end{align*}
Let $i\in \{1,2\}$ be fixed. 
Then Theorem~\ref{thmTS} implies that $x\in [0,1]$ is $T_i$-generic if and only if $x$ is $S_i$-generic. 
The graphs of $T_1$, $S_1$ and graphs of $T_2$, $S_2$ are shown in Figure~\ref{fig1} and \ref{fig2} respectively. 
\end{exam}

\begin{figure}
\begin{center}
\begin{tikzpicture}[scale=1]
	\node[below left] at (0, 0) {$0$};
	\draw[->] (-0.5, 0) -- (4.5, 0); 
	\draw[->] (0, -0.5) -- (0, 4.5); 
	\draw[very thick] (0, 0) -- (2,4);
	\draw[very thick] (2, 0) -- (4,4);

	\draw[dotted] (2, 4) -- (2,0) node[below] {$\frac 12$};
	\draw[dotted] (4, 4) -- (4,0) node[below] {$1$};

	\draw[dotted] (4, 4) -- (0,4) node[left] {$1$};

\end{tikzpicture}
\qquad
\begin{tikzpicture}[scale=1]
	\node[below left] at (0, 0) {$0$};
	\draw[->] (-0.5, 0) -- (4.5, 0); 
	\draw[->] (0, -0.5) -- (0, 4.5); 
	\draw[very thick] (0, 0) -- (2,4);
	\draw[very thick] (2, 0) -- (3,4);
	\draw[very thick] (3, 0) -- (4,4);

	\draw[dotted] (2, 4) -- (2,0) node[below] {$\frac 12$};
	\draw[dotted] (3, 4) -- (3,0) node[below] {$\frac 34$};
	\draw[dotted] (4, 4) -- (4,0) node[below] {$1$};

	\draw[dotted] (4, 4) -- (0,4) node[left] {$1$};
\end{tikzpicture}
\caption{The maps $T_1$ and $S_1$ in Example~\ref{Counter}\label{fig1}}
\end{center}
\end{figure}

\begin{figure}
\begin{center}
\begin{tikzpicture}[scale=1]
	\node[below left] at (0, 0) {$0$};
	\draw[->] (-0.5, 0) -- (4.5, 0);
	\draw[->] (0, -0.5) -- (0, 4.5); 
	\draw[very thick] (0, 0) -- (2.472,4);
	\draw[very thick] (2.472, 0) -- (4,2.472);

	\draw[dotted] (2.472, 4) -- (2.472,0) node[below] {$1/\beta$};
	\draw[dotted] (4, 4) -- (4,0) node[below] {$1$};

	\draw[dotted] (4, 4) -- (0,4) node[left] {$1$};
	\draw[dotted] (4, 2.472) -- (0,2.472) node[left] {$\beta-1$};
\end{tikzpicture}
\qquad
\begin{tikzpicture}[scale=1]
	\node[below left] at (0, 0) {$0$};
	\draw[->] (-0.5, 0) -- (4.5, 0);
	\draw[->] (0, -0.5) -- (0, 4.5); 
	\draw[very thick] (0, 0) -- (2.472,4);
	\draw[very thick] (2.472, 0) -- (4,4);

	\draw[dotted] (2.472, 4) -- (2.472,0) node[below] {$1/\beta$};
	\draw[dotted] (4, 4) -- (4,0) node[below] {$1$};

	\draw[dotted] (4, 4) -- (0,4) node[left] {$1$};
\end{tikzpicture}
\caption{The maps $T_2$ and $S_2$ in Example~\ref{Counter}\label{fig2}}
\end{center}
\end{figure}
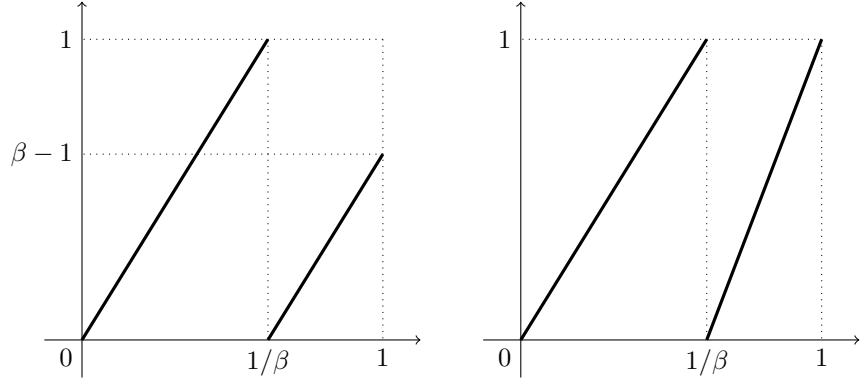

Examples \ref{FlipDecimal} and \ref{Counter} are generalized as follows.

\begin{exam}\label{Luroth}
Let $\beta$ be a Perron number: an algebraic integer greater than one whose conjugates
have modulus less than $\beta$. Handelman \cite{Han}
showed that $\beta$ has no other positive conjugates if and only if
there exist an $\ell\in \N$ and a 
nonnegative integer vector $(a_1,\dots, a_{\ell})$ satisfying 
$$
1=\sum_{i=1}^{\ell} \frac {a_i}{\beta^i}.
$$
If there exists such a vector, then there are infinitely many different expressions of $1$ 
of this form.
Assume further that $\beta$ is a Pisot number having no other positive conjugate. 
For such a vector $(a_1,\dots, a_{\ell})$ we can
partition $[0,1]$ into $a_1+\dots+a_{\ell}$ sub-intervals; $a_i$ intervals of length $\beta^{-i}$, arranged in arbitrary order, and construct a piecewise linear transformation $S$
of slopes $\pm \beta^{i}$ for $i=1,\dots,\ell$ whose all discontinuities are mapped to $\{0,1\}$.
The invariant measure of $S$ is the Lebesgue measure. 
All the maps $S$ produced from a fixed Pisot number $\beta$ in this manner are normality equivalent, because all of them are generic point equivalent to $T(x)=\{\beta x\}$ by Theorem~\ref{thmTS}.
\end{exam}

\begin{exam}
\label{ACIM}
Take a real number $\beta>1$ and $t\in [0,\lceil \beta \rceil/\beta-1]$. Define a map $S_t:[0,1]\to [0,1]$ by
\begin{align*}
S_t(x):=\begin{cases}
\beta x -\lfloor \beta x\rfloor, & x\in [0,\lfloor \beta \rfloor/\beta-t),\\
\beta (x-1)+1 & x\in [\lfloor \beta \rfloor/\beta-t,1].
\end{cases}
\end{align*}
See Figure~\ref{fig3} for the graphs of $S_t$ for some $t$.  
As the map $S_t$ has only one non trivial discontinuity at $r_0=l_0=\lfloor \beta \rfloor/\beta -t$,
it is ergodic with respect to a unique absolutely continuous invariant measure
(c.f. \cite{LiYorke}).
Its invariant density is made explicit as 
$$
h(x)=C + \sum_{x\ge r_n} \frac1{\beta^n} + \sum_{ x < l_n} \frac1{\beta^n},
$$
where the sums are taken over positive integers $n$. Here 
$r_n=S_t^n(\lfloor \beta \rfloor/\beta -t+0)$ and
$l_n=S_t^n(\lfloor \beta \rfloor/\beta -t-0)$.
The constant $C$ is computed as
$$
C=\frac {\beta-2}{\beta-1}+ \sum_{n=1}^{\infty} \frac{ \iota^+(n)-\iota^{-}(n)}{\beta^n}
$$
with
$$
\iota^{+}(n)=\begin{cases} 1 & r_n \ge r_0\\
                           0 & r_n < r_0
\end{cases}
\qquad \mbox{and} \qquad
\iota^{-}(n)=\begin{cases} 1 & l_n \ge l_0\\
                           0 & l_n < l_0.
\end{cases}
$$
Though $C$ can be negative, we claim for any pair $(\beta, t)$ that 
\begin{itemize}
\item[(*)]There exists a positive $c$ that $c^{-1}<h(x)<c$ if and only if $\beta\ge 2$. 
\end{itemize} 
Its proof is given in the appendix. 
Moreover, we shall show that $c$ depends only on $\beta$. 

Hence, we see that if $\beta$ is a Pisot number not less than 2, then the map $S$ satisfies the assumptions in Theorem~\ref{thmTS}. 
Therefore, if $\beta\not \in \Z$, then all maps in the one parameter family with cardinality of continuum 
$$
\left\{ S_t \ \left|\ t\in\R, \ 0\le t\le \frac{\lceil \beta \rceil}{\beta}-1 \right.\right\}
$$
are generic point equivalent by Theorem~\ref{thmTS}.
\end{exam}

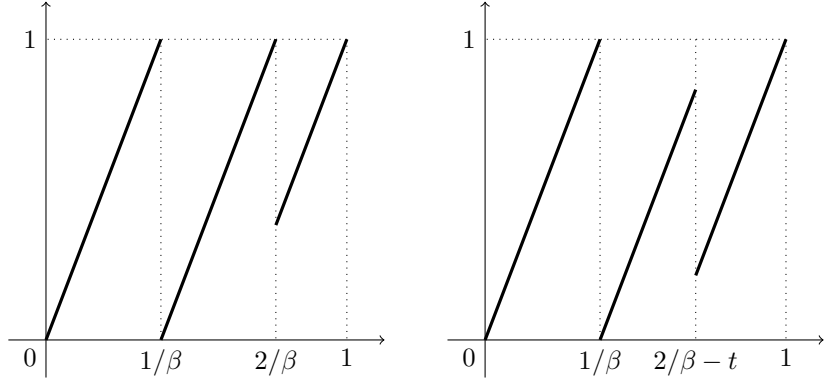
\begin{figure} 
\begin{center}
\begin{tikzpicture}[scale=1]
	\node[below left] at (0, 0) {$0$};
	\draw[->] (-0.5, 0) -- (4.5, 0);
	\draw[->] (0, -0.5) -- (0, 4.5); 
	\draw[very thick] (0, 0) -- (1.528,4);
	\draw[very thick] (1.528, 0) -- (3.056,4);
	\draw[very thick] (3.056, 1.528) -- (4, 4);

	\draw[dotted] (1.528, 4) -- (1.528,0) node[below] {$1/\beta$};
	\draw[dotted] (3.056, 4) -- (3.056,0) node[below] {$2/\beta$};
	\draw[dotted] (4, 4) -- (4,0) node[below] {$1$};

	\draw[dotted] (4, 4) -- (0,4) node[left] {$1$};
\end{tikzpicture}
\qquad
\begin{tikzpicture}[scale=1]
	\node[below left] at (0, 0) {$0$};
	\draw[->] (-0.5, 0) -- (4.5, 0);
	\draw[->] (0, -0.5) -- (0, 4.5); 
	\draw[very thick] (0, 0) -- (1.528,4);
	\draw[very thick] (1.528, 0) -- (2.8,3.3305);
	\draw[very thick] (2.8, 0.8585) -- (4, 4);

	\draw[dotted] (1.528, 4) -- (1.528,0) node[below] {$1/\beta$};
	\draw[dotted] (2.8, 4) -- (2.8,0) node[below] {$2/\beta - t$};
	\draw[dotted] (4, 4) -- (4,0) node[below] {$1$};

	\draw[dotted] (4, 4) -- (0,4) node[left] {$1$};
\end{tikzpicture}
\caption{The maps $S_0$ and $S_t$ in Example~\ref{ACIM} for $\beta = (3+\sqrt 5)/2$\label{fig3}}
\end{center}
\end{figure}

\section*{Acknowledgments}

The authors would like to thank Michihiro Hirayama for giving fruitful advice. 
This research was partially supported by JSPS grants (17K05159, 17H02849, BBD30028, 15K17505, 19K03439) and NRF of Korea (NRF-2018R1A2B6001624).

\appendix
\section{Positivity of invariant density}
To study the invariant densities of a piecewise linear map, a general method is established 
by Kopf \cite{Kopf} and Gora \cite{Gora}.  It works well for a given map. 
To deal with the parametrized family 
of maps in Example \ref{ACIM}, we follow an analogy of Parry \cite{Par, Par2} to calculate the invariant density and deduce the claim (*). 
For simplicity, we write $S=S_t$.
When $\beta<2$, the map $S$ is dissipative in $Y:=[0,r_1) \cup
[l_1,1)$ and $h(x)=0$ in $Y$. For an integer $\beta>1$, the map $S$ is the
$\beta$-adic transformation and  preserve the Lebesgue measure.
Therefore we have to show that $h(x)$ is positive for $\beta>2$ and
$\beta\not \in \Z$. 
Putting
$$
d_n^+(x)=\begin{cases} 1& x\ge r_n\\ 0 & x<r_n\end{cases}
\qquad \mbox{and} \qquad
d_n^-(x)=\begin{cases} 1& x< l_n\\ 0 & x\ge l_n\end{cases}
$$
for $n=1,2,\dots$, we see that $h(x)=C + \sum_{n=1}^{\infty} d_n(x)/\beta^n$
with $d_n(x):=d_n^+(x)+d_n^-(x)$. Define the digit $\alpha(x):=\beta x-S(x)$ for $x\in [0,1)$.
Then 
$$\mathcal{D}=\{\alpha(x) |\ x\in [0,1)\}= 
\{0,1,\dots, \lfloor \beta\rfloor-1\} \cup \{\beta-1\}.$$ 
Put $$
\mathcal{D}(x)=
\begin{cases} \mathcal{D} \setminus \{\beta-1\}, & x\in [0,r_1), \\
 \mathcal{D} & x\in [r_1,l_1),\\
 \mathcal{D}\setminus \{\lfloor \beta\rfloor -1\}, & x\in [l_1,1),\\
\end{cases}
$$
$$
e_n^+(x)=\mathrm{Card}\{ d\in \mathcal{D}(x)\ |\ d > \alpha(r_n)\} 
-\begin{cases} 1, & \alpha(r_n)=\lfloor \beta \rfloor -1 \ \mathrm{ and }\  x \ge l_1,\\
               0, & \mathrm{otherwise} \end{cases}
$$
and
$$
e_n^-(x)=\mathrm{Card}\{ d\in \mathcal{D}(x)\ |\ d < \alpha(l_n)\} 
-\begin{cases} 1, & \alpha(l_n)=\beta -1 \ \mathrm{ and }\  x < r_1,\\
               0, & \mathrm{otherwise}. \end{cases}
$$
Then we observe the key equality:
$$
\sum_{y\in S^{-1}( \{x\} )} d_n(y)= e_n(x)+d_{n+1}(x)
$$
with $e_n(x):=e_n^+(x)+e_n^{-}(x)$. Therefore 
\begin{align*}
\frac 1{\beta} \sum_{y\in S^{-1}( \{x\} )} h(y) 
&= \frac 1{\beta} \sum_{y\in S^{-1}( \{x\} )} C
+ \sum_{n=1}^{\infty} \frac 1{\beta^{1+n}} \sum_{y\in S^{-1}( \{x\} )} d_n(y) \\
&= \frac {(\lfloor \beta \rfloor-1+d_1(x))C}{\beta}
+ \sum_{n=1}^{\infty} \frac {e_n(x)+d_{n+1}(x)}{\beta^{n+1}}.
\end{align*}
To be an invariant density, we have to show that this is nothing but $h(x)$. 
It is sufficient to confirm:
$$
C \left(1- \frac {\lfloor \beta \rfloor -1}{\beta}-\frac {d_1(x)}{\beta}
\right) = \sum_{n=1}^{\infty} \frac{e_n(x)}{\beta^{n+1}}-\frac {d_1(x)}{\beta}.
$$
We can check that the integration over $[0,1]$ of both sides vanishes. Moreover, both sides 
take only two values, i.e., they are constant in $[0,r_1)\cup [l_1,1)$ and in $[r_1,l_1)$.
This shows the existence of a constant $C$. Computation of $C$ is therefore done at any point $x$ in $[0,1)$. 
Evaluating  at $x=0$, we have
$e_n^+(0) = \lfloor \beta \rfloor -1- \lfloor \alpha (r_n) \rfloor$ and $e_n^-(0) = \lfloor \alpha (l_n) \rfloor$.
Then we apply
$$
r_0=\sum_{n=1}^{\infty} \frac{\alpha(r_{n-1})}{\beta^{n}}, \quad
 l_0=\sum_{n=1}^{\infty} \frac{\alpha(l_{n-1})}{\beta^{n}}
$$
to obtain
$$
C=1-\frac 1{\beta-1} + \sum_{n=1}^{\infty} \frac {\iota^{+}(n)-\iota^-(n)}{\beta^n}
$$
and
\begin{equation*}
h(x)=1 + \sum_{n=1}^{\infty} \frac {\iota^{+}(n)-\iota^-(n)}{\beta^n} +
\sum_{n=1}^{\infty} \frac{d_n(x)-1}{\beta^n}.
\end{equation*}
We have $\iota^{+}(n)-\iota^-(n)\in \{-1,0,1\}$ and $d_n(x)-1\in \{-1,0,1\}$.
Note that $\iota^+(n)-\iota^-(n)=-1$ if and only if $r_n<\lfloor \beta\rfloor/\beta-t\le l_n$,
and $r_n<l_n$ implies $d_n(x)\ge 1$. 
Moreover,
$\iota^+(n)-\iota^-(n)=1$ if and only if $l_n<\lfloor \beta\rfloor/\beta-t\le r_n$, 
and $l_n<r_n$ implies $d_n(x)\le 1$. Therefore we obtain
$$
\iota^+(n)-\iota^-(n)+d_n(x)-1\in \{-1,0,1\}
$$
and
$$
\frac {\beta-2}{\beta-1} \le h(x)\le \frac{\beta}{\beta-1}.
$$

\begin{thebibliography}{99}
\bibitem{AJKM} D.~Airey, S.~Jackson, D.~Kwietniak and B.~Mance, Borel complexity of sets of normal numbers
via generic points in subshifts with specification, ArXiv:1811.04450v1.

\bibitem{Gar} A.~M. Garsia.  Arithmetic properties of {B}ernoulli
 convolutions.  {\em Trans. Amer. Math. Soc.}, {\bf 102}  (1962), 409--432.
\bibitem{Gora} P.~G\'ora, Invariant densities for piecewise linear maps of the unit interval.
{\em Ergodic Theory Dynam. Systems}, {\bf 29} (2009), no. 5, 1549--1583. 
\bibitem{Gora0} P.~G\'ora,  Invariant densities for generalized $\beta$-maps. 
{\em Ergodic Theory Dynam. Systems}, {\bf 27} (2007), no. 5, 1583--1598. 
\bibitem{Han}
D.~Handelman, 
Spectral radii of primitive integral companion matrices and log concave polynomials. Symbolic dynamics and its applications (New Haven, CT, 1991), 231--237, 
Contemp. Math., 135, Amer. Math. Soc., Providence, RI, 1992.
\bibitem{IT} S. Ito and Y. Takahashi, 
Markov subshifts and realization of $\beta$-expansions,
{\em J. Math. Soc. Japan}, 26-1 (1974), 33--55.
\bibitem{Jag}
H. J\"ager, On decimal expansions, in ``Zahlentheorie (Tagung), Math.
Forschungsinst. Oberwolfach, 1970,'' pp. 67--75, Bereich Math. Forschungsinst.,
Oberwolfach, Heft 5, Bibliographisches Inst., Mannheim, 1971.
\bibitem{JV}
S.~Jung and B.~Volkmann, Remarks on a paper of Wagner.
{\em J. Number Theory} {\bf 56} (1996), no. 2, 329--335.

\bibitem{KL} H.~Ki and T.~Linton, Normal numbers and subset of {N} with given densities, {\em Fund. Math.} {\bf 144}, (1994), no.2. 163--179.

\bibitem{KS}
H.~Kano and I.~Shiokawa, Rings of normal and nonnormal numbers. 
{\em Israel J. Math.} {\bf 84} (1993), no. 3, 403--416. 
\bibitem{Kopf}
C.~Kopf, Invariant measures for piecewise linear transformations of the interval. 
{\em Appl. Math. Comput.}, {\bf 39} (1990), no. 2, part II, 123--144. 

\bibitem{Kow}
Z.~Kowalski, 
Invariant measure for piecewise monotonic transformation has a positive lower bound on its support.
{\em Bull. Acad. Polon. Sci. Ser. Sci. Math.}, {\bf 27} (1979), no. 1, 53--57. 

\bibitem{KN}
C.~Kraaikamp and H.~Nakada, On a Problem of Schweiger concerning normal numbers, 
{\em J. Number Theory}, {\bf 86} (2001), 330--340.

\bibitem{LiYorke}
T.~Y.~Li and J.~A.~Yorke, Ergodic transformations from an interval into itself. 
{\em Trans. Amer. Math. Soc.}, {\bf 235} (1978), 183--192. 

\bibitem{Max}
J. E. Maxfield, Normal k-tuples, {\em Pacific J. Math.}, {\bf 3} (1953), 189--196.

\bibitem{MosShk}
N. G. Moshchevitin and I. D. Shkredov, 
On the Pyatetskii-Shapiro criterion of normality, 
{\em Math. Notes} {\bf 73} (2003), 539--550. 

\bibitem{Par} W.~Parry, On the {$\beta $}-expansions of real numbers,
{\em Acta Math. Acad. Sci. Hungar.}, {\bf 11} (1960), 401--416.

\bibitem{Par2} W.~Parry, Representations for real numbers. 
{\em Acta Math. Acad. Sci. Hungar.}, {\bf 15} (1964) 95--105. 

\bibitem{Poll} A. D. Pollington, 
The Hausdorff dimension of a set of normal numbers. {\em Pacific J. Math.}, {\bf 95} (1981), 193--204. 

\bibitem{Pos}
A. G. Postnikov, 
Ergodic problems in the theory of congruences and of diophantine approximations [in Russian], 
{\em Trudy Mat. Inst. Steklov}, {\bf 82} (1966); 
English trans. {\em Proc. Steklov Inst. Math.} {\bf 82}, 
American Math. Society, Providence, R. I. 1967.

\bibitem{WSch}
W.~M.~Schmidt, On normal numbers, {\em Pacific J. Math.}, {\bf 10} (1960) 661--672.
\bibitem{Sch}
F. Schweiger,
 Normalit\"at bez\"uglich zahlentheoretischer Transformationen.  {\em J. Number Theory}, {\bf 1} (1969), 390--397.
\bibitem{SS} A.~N.~Sharkovsky and A.~G.~Sivak, Basin of attractors of trajectories, {\em J. Diff. Eqn. Appl.}, {\bf 22}, (2016), no. 2, 159--163.  
 \bibitem{Van}
J. Vandehey, On the joint normality of certain digit expansions, arXiv:1408.0435.
\bibitem{Wag}
G.~Wagner, On rings of numbers which are normal to one base but non-normal to another. 
{\em J. Number Theory} {\bf 54} (1995), no. 2, 211--231. 
\bibitem{Wal}
D. D. Wall, Normal numbers, Ph. D. thesis (1949), University of California, Berkeley.
\end{thebibliography}
\end{document}